\documentclass[a4paper,12pt]{amsart}

\usepackage[utf8]{inputenc}
\usepackage{amsfonts,amssymb}
\usepackage{amsmath}
\usepackage{graphicx}

\usepackage{comment}
\usepackage{color}
\usepackage{tikz}
\usetikzlibrary{topaths}

\usepackage{a4wide}

\usepackage{empheq}
\numberwithin{equation}{section}

\usepackage{amsthm}
\theoremstyle{plain}

\newtheorem{theorem}{Theorem}[section]
\newtheorem{lemma}[theorem]{Lemma}
\newtheorem{corollary}[theorem]{Corollary}
\newtheorem{proposition}[theorem]{Proposition}
\newtheorem{observation}[theorem]{Observation}
\newtheorem{cit}[theorem]{Citation}

\newtheorem*{main:separation}{Theorem~\ref{thrm:separation}}
\newtheorem*{main:pb_sigma_ceiling}{Theorem~\ref{thrm:pb_sigma_ceiling}}

\theoremstyle{definition}
\newtheorem{definition}[theorem]{Definition}
\newtheorem{remark}[theorem]{Remark}

\newcommand{\R}{\mathbb{R}}
\newcommand{\N}{\mathbb{N}}
\newcommand{\Z}{\mathbb{Z}}

\newcommand{\defeq}{\mathrel{\mathop{:}}=}
\newcommand{\lk}{\operatorname{lk}}
\newcommand{\st}{\operatorname{st}}
\newcommand{\Hom}{\operatorname{Hom}}
\newcommand{\Aut}{\operatorname{Aut}}

\newcommand{\id}{\operatorname{id}}
\newcommand{\asta}{{\st}{\uparrow}}
\newcommand{\dst}{{\st}{\downarrow}}
\newcommand{\alk}{{\lk}{\uparrow}}
\newcommand{\dlk}{{\lk}{\downarrow}}
\newcommand{\palk}{{\operatorname{plk}}{\uparrow}}
\newcommand{\cross}{\kappa}

\DeclareMathOperator{\W}{W}
\DeclareMathOperator{\PW}{PW}
\DeclareMathOperator{\Rev}{Rev}
\DeclareMathOperator{\F}{F}

\numberwithin{equation}{section}

\begin{document}

\title{Separation in the BNSR-invariants of the pure braid groups}

\date{\today}


\author{Matthew C.~B.~Zaremsky}
\address{Department of Mathematical Sciences, Binghamton University, Binghamton, NY 13902}
\email{zaremsky@math.binghamton.edu}

\begin{abstract}
 We inspect the BNSR-invariants $\Sigma^m(P_n)$ of the pure braid groups $P_n$, using Morse theory. The BNS-invariants $\Sigma^1(P_n)$ were previously computed by Koban, McCammond and Meier. We prove that for any $3\le m\le n$, the inclusion $\Sigma^{m-2}(P_n)\subseteq \Sigma^{m-3}(P_n)$ is proper, but $\Sigma^\infty(P_n)=\Sigma^{n-2}(P_n)$. We write down explicit character classes in each relevant $\Sigma^{m-3}(P_n)\setminus \Sigma^{m-2}(P_n)$. In particular we get examples of normal subgroups $N\le P_n$ with $P_n/N\cong\Z$ such that $N$ is of type $\F_{m-3}$ but not $\F_{m-2}$, for all $3\le m\le n$.
\end{abstract}

\maketitle
\thispagestyle{empty}


\section*{Introduction}

The $\Sigma$-invariants of a group $G$ are a sequence of geometric objects $\Sigma^m(G)$ ($m\in\N$), defined whenever $G$ is of type $\F_m$, that encode a great deal of information about $G$ and its subgroups. The first invariant, $\Sigma^1$, was introduced by Bieri, Neumann and Strebel \cite{bieri87} and is also called the BNS-invariant. The higher invariants, $\Sigma^2$, $\Sigma^3$, and so forth, culminating in $\Sigma^\infty$, were subsequently introduced by Bieri and Renz \cite{bieri88}, and are also known as BNSR-invariants. Once one knows the BNSR-invariants of a group, one gets a complete classification of which of its coabelian subgroups have which finiteness properties. (A \emph{coabelian} subgroup is a normal subgroup with abelian quotient.)

Once the finiteness properties of a group are known (e.g., whether it is finitely generated, finitely presented, type $\F_m$, etc.), the group's BNSR-invariants are a very natural next question. However, in general the BNSR-invariants of a group are notoriously difficult to compute. A complete computation has been done only for very few families of ``interesting'' groups. The main example where the problem is interesting, difficult, and totally solved is the case of right-angled Artin groups, whose BNSR-invariants were computed by Meier, Meinert and VanWyk \cite{meier98} and, independently, by Bux and Gonzalez \cite{bux99}. There are also some results for general Artin groups \cite{meier01}. The Morse-theoretic approach we take here is related to the methods used in \cite{witzel15,zaremsky15} to compute BNSR-invariants of generalizations of Thompson's groups.

The braid groups $B_n$ and pure braid groups $P_n$ are extremely well studied families of groups. They arise in a variety of contexts, for example in the study of Artin groups, mapping class groups and knot theory. The braid group $B_n$ is of type $\F$, meaning it admits a compact classifying space, so from the point of view of finiteness properties $B_n$ is as ``good'' as it can possibly be. Since $P_n$ has finite index in $B_n$, it also has this property. From the point of view of BNSR-invariants though, the pure braid groups are vastly more complicated than the braid groups. The abelianization of $B_n$, for any $n\ge2$, is just $\Z$, so the $\Sigma^m(B_n)$ all live in a $0$-sphere, and it is easy to show that the whole $0$-sphere equals $\Sigma^\infty(B_n)$ (Corollary~\ref{cor:braid_sig}). In particular, every coabelian subgroup of $B_n$ is of type $\F_\infty$. The abelianization of $P_n$, however, is $\Z^{\binom{n}{2}}$, so the $\Sigma^m(P_n)$ all lie in an $(\binom{n}{2}-1)$-sphere. With such a huge sphere, figuring out which parts of it do or do not lie in which BNSR-invariants is massively more complicated. Correspondingly, it is much more difficult to figure out the finiteness properties of coabelian subgroups of $P_n$.

The first $\Sigma$-invariant, that is the BNS-invariant $\Sigma^1(P_n)$, was computed by Koban, McCammond and Meier \cite{koban15}. In particular for $n\ge3$ there are parts of the character sphere that do not lie in $\Sigma^1(P_n)$, and hence there are coabelian subgroups of $P_n$ (for example the commutator subgroup itself) that are not finitely generated, in contrast to the $B_n$ case. Since $P_3 \cong F_2\times\Z$, every finitely generated coabelian subgroup is already finitely presented and even of type $\F_\infty$. For $n>3$ though, it was unclear whether there exist coabelian subgroups that are finitely generated but not finitely presentable.

In this paper, we use Morse theory to find regions of the character sphere of $P_n$ that lie in $\Sigma^m(P_n)$ for various $m$, and regions that do not. Our main results are:

\begin{main:separation}
 For any $3\le m\le n$, the inclusion $\Sigma^{m-2}(P_n)\subseteq \Sigma^{m-3}(P_n)$ is proper.
\end{main:separation}

\begin{main:pb_sigma_ceiling}
 For any $n\in\N$, $\Sigma^\infty(P_n)=\Sigma^{n-2}(P_n)$.
\end{main:pb_sigma_ceiling}

Consequences of this include that there exist coabelian subgroups of $P_n$, for any $n\ge4$, that are finitely generated but not finitely presentable. More generally there exist coabelian subgroups of $P_n$ that are of type $\F_{m-3}$ but not $\F_{m-2}$, for any $3\le m\le n$. We give explicit examples of these; see Corollary~\ref{cor:separation}. However, a coabelian subgroup of $P_n$ of type $\F_{n-2}$ is automatically of type $\F_\infty$.

Our approach reduces the problem of proving a character class is or is not in a given BNSR-invariant of $P_n$ to a sufficient condition on a finite complex $\PW_n$. Hence it seems likely that our approach could reveal still more about the $\Sigma^m(P_n)$ in the future. We mention that even the problem of fully computing $\Sigma^2(P_4)$ (which equals $\Sigma^\infty(P_4)$) remains open.

This paper is organized as follows. In Section~\ref{sec:invs} we recall the BNSR-invariants and discuss some general results, and in Section~\ref{sec:morse} we set up our Morse-theoretic approach. In Section~\ref{sec:gps} we discuss the (pure) braid groups and their characters, and in Section~\ref{sec:cpx} a complex on which they act. Finally in Section~\ref{sec:results} we use Morse theory on this complex to derive our results about the BNSR-invariants of the pure braid groups.

\subsection*{Acknowledgments} I am grateful to Robert Bieri for many helpful discussions and comments, and in particular for pointing out some references that simplified many things in Section~\ref{sec:invs}.


\section{BNSR-invariants}\label{sec:invs}

To define the BNSR-invariants $\Sigma^m(G)$ of a group $G$, we first need some background. A connected CW-complex $Z$ is called a \emph{classifying space} for $G$, or $K(G,1)$, if $\pi_1(Z)\cong G$ and $\pi_k(Z)=0$ for $k\ge2$. We say $G$ is of \emph{type $\F_n$} if it admits a $K(G,1)$ with compact $n$-skeleton. If $G$ is of type $\F_n$ for all $n$ we say it is of \emph{type $\F_\infty$}. If there exists a compact $K(G,1)$ we say $G$ is of \emph{type $\F$}.

Now suppose $G$ is a group of type $\F_n$ and let $Z$ be a $K(G,1)$ with compact $n$-skeleton. The universal cover $\widetilde{Z^{(n)}}$ of the $n$-skeleton is $(n-1)$-connected, and $G$ acts freely and cocompactly on it. In fact, there is a converse to this: if $G$ acts freely (or just properly) and cocompactly on an $(n-1)$-connected CW-complex, then $G$ is of type $\F_n$. This brings us to the definition of the BNSR-invariants.

\begin{definition}[BNSR-invariants]\label{def:sig_invs}
 Let $G$ act properly cocompactly on an $(n-1)$-connected CW-complex $Y$ (so $G$ is of type $\F_n$). Let $\chi\colon G\to\R$ be a \emph{character} of $G$, i.e., a homomorphism to $\R$. There exists a continuous map $h_\chi \colon Y\to\R$ such that $h_\chi(g.y)=\chi(g)+h_\chi(y)$ for all $g\in G$ and $y\in Y$. For $t\in\R$ let $Y_{\chi\ge t}$ be the full subcomplex of $Y$ supported on those vertices $y$ with $h_\chi(y)\ge t$. For non-trivial $\chi$, let $[\chi]$ be the equivalence class of $\chi$ under scaling by positive real numbers. For $m\le n$, the $m$th \emph{BNSR-invariant} $\Sigma^m(G)$ is defined to be
 $$\Sigma^m(G)\defeq \{[\chi]\mid (Y_{\chi\ge t})_{t\in\R} \text{ is essentially $(m-1)$-connected}\}\text{.}$$
\end{definition}

The definition of $\Sigma^m$ is admittedly quite dense. Let us unpack it a bit. The characters $\chi$ are elements of $\Hom(G,\R)$, which is a vector space $\R^d$ for some $d$. The equivalence classes $[\chi]$ lie in the so called sphere at infinity $S^{d-1}$ of $\R^d$, or \emph{character sphere} $S(G)$. A character $\chi$ induces a \emph{character height function} $h_\chi \colon Y \to \R$. The height function $h_\chi$ can be used to divide $Y$ into regions, for example the region whose vertices have height greater than or equal to $t$. By varying $t$, we get a nested family of these ``half spaces''. Now if, for example, each half space were itself $(m-1)$-connected, then the definition says $[\chi]$ would lie in $\Sigma^m(G)$. Heuristically, $\Sigma^m(G)$ is a catalog of which half spaces are how highly connected. However, this is not quite right, as the presence of the word ``essentially'' indicates. Instead, we do not require each half space to be $(m-1)$-connected itself, but only that for all $t$ there exists $-\infty<s\le t$ such that the inclusion $Y_{\chi\ge t}\to Y_{\chi\ge s}$ induces the trivial map in $\pi_k$ for $k\le m-1$. (If the domain is already $(m-1)$-connected, then of course this will be the case.) For example, maybe $Y_{\chi\ge t}$ is not connected, but if all its components get connected up in $Y_{\chi\ge t-1}$ then that is enough to say $[\chi]$ is in $\Sigma^1(G)$.

It is not obvious from the definition, but $\Sigma^m(G)$ is well defined up to the various non-canonical aspects, e.g., the space $Y$ and the height function $h_\chi$. See for example \cite[Definition~8.1]{bux04}. As a remark, the definition there used the filtration by sets $h_\chi^{-1}([t,\infty))_{t\in\R}$, but thanks to cocompactness this filtration is essentially $(m-1)$ connected if and only if our filtration $(Y_{\chi\ge t})_{t\in\R}$ is.

One important application of BNSR-invariants is the following:

\begin{cit}\cite[Theorem~1.1]{bieri10}\label{cit:sig_fin}
 Let $G$ be a group of type $\F_m$. Let $H$ be any coabelian subgroup of $G$. Then $H$ is of type $\F_m$ if and only if for every non-trivial character $\chi$ of $G$ such that $\chi(H)=0$, we have $[\chi]\in\Sigma^m(G)$.
\end{cit}

For example, if $H=\ker(\chi)$ for $\chi$ a \emph{discrete} character of $G$, i.e., one with image $\Z$, then $H$ is of type $\F_m$ if and only if $[\pm\chi]\in\Sigma^m(G)$.

Other important properties of the $\Sigma^m(G)$ are that they are all open subsets of $S(G)$, and that they are invariant under the natural action of $\Aut(G)$ on $S(G)$. We collect here some additional general results about BNSR-invariants that will be useful in the following sections.

\begin{cit}\cite[Corollary~2.8]{meinert97}\label{cit:push_neg}
 Let $G\stackrel{\pi}{\to} Q$ be a split epimorphism of groups. Let $\chi\colon Q \to \R$ be a character. Let $\widetilde{\chi} \colon G \to \R$ be $\widetilde{\chi}=\chi\circ \pi$. If $[\widetilde{\chi}]\in \Sigma^m(G)$ then $[\chi]\in \Sigma^m(Q)$.
\end{cit}

\begin{lemma}\label{lem:push_pos}
 Let $G\stackrel{\pi}{\to} Q$ be an epimorphism of groups. Suppose that $\ker(\pi)$ is of type $\F_m$. Let $\chi\colon Q \to \R$ be a character. Let $\widetilde{\chi} \colon G \to \R$ be $\widetilde{\chi}=\chi\circ \pi$. If $[\psi]\in \Sigma^m(Q)$ for all non-trivial characters $\psi$ of $Q$ such that $\psi(\ker(\chi))=0$, then $[\widetilde{\chi}]\in \Sigma^m(G)$.
\end{lemma}

\begin{proof}
 Consider the restriction $\pi'$ of $\pi$ to $\ker(\widetilde{\chi})$, so $\pi' \colon \ker(\widetilde{\chi}) \to \ker(\chi)$. Note that $\pi'$ is surjective and $\ker(\pi')=\ker(\pi)$, so we have a short exact sequence
 $$1 \to \ker(\pi) \to \ker(\widetilde{\chi}) \to \ker(\chi) \to 1 \text{.}$$
 We are assuming $[\psi]\in \Sigma^m(Q)$ for all non-trivial characters $\psi$ of $Q$ such that $\psi(\ker(\chi))=0$, so Citation~\ref{cit:sig_fin} says that $\ker(\chi)$ is of type $\F_m$. We are also assuming $\ker(\pi)$ is of type $\F_m$, so in fact $\ker(\widetilde{\chi})$ is of type $\F_m$. Again by Citation~\ref{cit:sig_fin}, we conclude that $[\widetilde{\chi}]\in \Sigma^m(G)$.
\end{proof}

This next result was called the $\Sigma^m$-criterion in \cite{meier01}; the homological version is Theorem~4.1 in \cite{bieri88}, and a proof of this homotopical version follows by mimicking that proof.

\begin{cit}\label{cit:raise_chi}
 With the notation from Definition~\ref{def:sig_invs}, $[\chi]\in\Sigma^m(G)$ if and only if there exists a continuous, cellular $G$-equivariant map $\varphi \colon Y^{(m)}\to Y^{(m)}$ satisfying $h_\chi(\varphi(y))>h_\chi(y)$ for all $y\in Y^{(m)}$.
\end{cit}

\begin{corollary}[Center survives]\label{cor:center_survives}
 Let $G$ be a group of type $\F_m$. For a character $\chi \colon G\to\R$, if $\chi(Z(G))\ne 0$ then $[\chi]\in\Sigma^m(G)$.
\end{corollary}

This is proved in \cite[Lemma~2.1]{meier01}, but the proof is so short we may as well recreate it here.

\begin{proof}
 Let $Y$ be an $(m-1)$-connected complex on which $G$ acts properly cocompactly. Choose $z\in Z(G)$ such that $\chi(z)>0$. Let $\varphi \colon Y^{(m)} \to Y^{(m)}$ be the map $y\mapsto z.y$. Since $z$ is central, this is $G$-equivariant, so by Citation~\ref{cit:raise_chi}, $[\chi]\in\Sigma^m(G)$.
\end{proof}


\section{Morse theory}\label{sec:morse}

When trying to compute the BNSR-invariants of a group, Bestvina--Brady Morse theory \cite{bestvina97} can be a useful tool. Provided one sets things up correctly, Morse theory can reduce a global statement, like ``this filtration is essentially $(m-1)$-connected,'' to local statements, like ``the `ascending' part of any vertex link is $(m-1)$-connected.'' In practice, the links may be finite simplicial complexes, which makes the latter question easier, in theory.

We will focus now on \emph{affine cell complexes}. An affine cell complex is the quotient of a disjoint union of euclidean polytopes by an equivalence relation mapping each polytope injectively into the complex, such that the images of the polytopes, called \emph{cells}, meet in faces (see \cite[Definition~I.7.37]{bridson99}). Each cell of an affine cell complex $Y$ carries an affine structure. The \emph{star} $\st_Y v$ of a vertex $v$ in $Y$ is the subcomplex of $Y$ consisting of cells that are faces of cells containing $v$. The \emph{link} $\lk_Y v$ of $v$ is the set of directions out of $v$ into $\st_Y v$. The link is a simplicial complex, whose simplices consist of such directions into a given cell.

In \cite{bestvina97}, Bestvina and Brady defined a \emph{Morse function} on an affine cell complex $Y$ to be a map $Y\to \R$ that is affine on cells, takes discretely many values on the vertices, and is non-constant on edges. When using Morse theory to compute BNSR-invariants, one often needs some more subtlety, for example if one is dealing with non-discrete character height functions. Our definition of \emph{Morse function} here is similar to the one in \cite{witzel15,zaremsky15}:

\begin{definition}[Morse function]\label{def:morse_fxn}
 Let $Y$ be an affine cell complex and let $(h,f) \colon Y \to \R\times \R$ be a map such that the restrictions of $h$ and $f$ to any cell are affine functions. Suppose that there exists $\varepsilon>0$ such that for any pair of adjacent vertices $v$ and $w$, either $|h(v) - h(w)|\ge \varepsilon$, or else $h(v)=h(w)$ and $f(v)\ne f(w)$. If the set $f(Y^{(0)})$, considered under the usual total ordering of $\R$, has the property that any subset has a maximal element (that is, if the total ordering on $f(Y^{(0)})$ is an inverse well ordering), then we call $(h,f)$ an \emph{ascending-type Morse function}. If the set $f(Y^{(0)})$ with the usual total ordering induced by $\R$ has the property that any subset has a minimal element (so it is a well ordering), then we call $(h,f)$ a \emph{descending-type Morse function}. By a \emph{Morse function} we mean either an ascending-type or descending type Morse function.
\end{definition}

This requirement that the total ordering on $f(Y^{(0)})$ be an (inverse) well ordering is different from the definition of Morse function in \cite{witzel15,zaremsky15}. There we just required $f(Y^{(0)})$ to be finite, which of course implies the present definition. Here we need this more robust definition, since our $f$ (the function $\dot{\omega}$ in Section~\ref{sec:results}) will take infinitely many values.

\begin{remark}[Classical Morse function]\label{rmk:classical_morse}
 We will occasionally also deal with functions that are Morse functions in the sense of Bestvina--Brady, which in our language is a Morse function $(h,f)$ where $f$ is constant, say $0$, and $h(Y^{(0)})$ is a discrete subset of $\R$.
\end{remark}

Given a Morse function $(h,f)$ on an affine cell complex $Y$, one can define the important notion of \emph{ascending link}. First note that, on any cell $c$, $(h,f)$ achieves its maximum and minimum values on unique vertices of $c$. This is because $h$ and $f$ are affine on $c$, and are non-constant on all faces with positive dimension. Here we always use the lexicographic ordering on $\R\times\R$ to compare $(h,f)$-values. Define the \emph{ascending star} $\asta v$ of a vertex $v$ in $Y$ to be the subcomplex of $\st_Y v$ consisting of those cells on which $(h,f)$ achieves its minimum at $v$. Similarly the \emph{descending star} $\dst v$ is the subcomplex of cells whose $(h,f)$-maximum is at $v$. Now define the \emph{ascending link} $\alk v$ to be the subcomplex of $\lk_Y v$ consisting of those directions out of $v$ into $\asta v$. The \emph{descending link} $\dlk v$ is defined analogously. When all the cells are simplices, 
$\lk_Y v$ can be identified with the subcomplex of $\st_Y v$ consisting of those simplices not containing $v$, with $\alk v$ and $\dlk v$ similarly viewed as subcomplexes.

\begin{lemma}[Morse Lemma]\label{lem:morse}
 Let $Y$ be an affine cell complex and $(h,f) \colon Y \to \R\times \R$ an ascending-type Morse function. For $t\in\R$ let $Y_{h\ge t}$ be the subcomplex of $Y$ supported on vertices $v$ with $h(v)\ge t$. Let $s<t$ (allow the case $s=-\infty$). If for all vertices $v$ with $s\le h(v)<t$ the ascending link $\alk v$ with respect to $(h,f)$ is $(m-1)$-connected, then the inclusion $Y_{h\ge t} \to Y_{h\ge s}$ induces an isomorphism in $\pi_k$ for $k\le m-1$ and a surjection in $\pi_m$.
 
 Now suppose instead that $(h,f)$ is a descending-type Morse function. Let $Y_{h\le t}$ be the subcomplex supported on vertices $v$ with $h(v)\le t$. Let $s>t$ (allow $s=\infty$). If for all vertices $v$ with $t<h(v)\le s$ the descending link $\dlk v$ with respect to $(h,f)$ is $(m-1)$-connected, then the inclusion $Y_{h\le t} \to Y_{h\le s}$ induces an isomorphism in $\pi_k$ for $k\le m-1$ and a surjection in $\pi_m$.
\end{lemma}

\begin{proof}
 The ``descending'' version can be converted into the ``ascending'' version by replacing $(h,f)$ with $(-h,-f)$, so we just need to prove the ascending version. The proof should be compared to the proofs in \cite{bestvina97} of Lemma~2.5 and Corollary~2.6.
 
 If we can prove the result assuming $t-s\le \varepsilon$, then we will be done, thanks to induction (and compactness of spheres, for the $s=-\infty$ case). Let $Y_{s\le h<t}$ be the subcomplex supported on vertices $v$ with $s\le h(v) <t$. In particular, $Y_{h\ge s}$ is obtained from $Y_{h\ge t}$ by attaching, in some order, the vertices of $Y_{s\le h<t}$ along their relative links. The order needs to be an inverse well order, so that the set of vertices not yet attached has a unique maximum, i.e., a well defined ``next'' vertex to attach. Let $\prec$ be any inverse well order on the vertex set of $Y_{s\le h<t}$ satisfying the property that if $f(v)<f(w)$ then $v\prec w$. Since $(h,f)$ is an ascending-type Morse function, the set of $f$-values on vertices is itself inversely well ordered, so such a $\prec$ does indeed exist. When we write $v\prec w$, we mean that $w$ gets attached before $v$, so $v$ can ``see'' $w$ but not vice versa.
 
 The goal is now to show that when we attach a vertex, we do so along an $(m-1)$-connected relative link, so this induces an isomorphism in $\pi_k$ for $k\le m-1$ and a surjection in $\pi_m$, by the Mayer--Vietoris and Seifert--van Kampen Theorems. Then by transfinite induction, which we can use since $\prec$ is an inverse well order, we will get that the inclusion $Y_{h\ge t} \to Y_{h\ge s}$ also induces these types of maps, and we will be done.
 
 When we attach the vertex $v$, we do so along a relative link $\lk_{rel} v$ equal to the full subcomplex of $\lk_Y v$ supported on those vertices $w$ with either $h(w)\ge t$, or else $s\le h(w)<t$ and $v \prec w$. We now claim that this is exactly the ascending $\alk v$. First note that for any $w$ adjacent to $v$, Definition~\ref{def:morse_fxn} says either $|h(v) - h(w)| \ge \varepsilon$, or else $h(v)=h(w)$ and $f(v)\ne f(w)$. For $w$ of the first form, since $t-s\le \varepsilon$ we know $w$ cannot lie in $Y_{s\le h<t}$. Hence $h(w)>h(v)$ if and only if $h(w)\ge t$, so such a $w$ lies in $\alk v$ if and only if it lies in $\lk_{rel} v$. Now suppose $w$ is of the second form, so $h(v)=h(w)$ and $f(v)\ne f(w)$. In particular $s\le h(w)<t$. Now $w$ is in $\alk v$ if and only if $f(v)<f(w)$, and $w\in\lk_{rel} v$ if and only if $v\prec w$, so we need to show that for such $w$, $f(v)<f(w)$ if and only if $v\prec w$. This follows from our construction of $\prec$, since we know $f(v)\ne f(w)$.
 
 Having shown $\lk_{rel} v = \alk v$, we now know that $Y_{h\ge s}$ is obtained from $Y_{h\ge t}$ by coning off the ascending links of vertices in $Y_{s\le h<t}$. By assumption, these are all $(m-1)$-connected, and so we are done.
\end{proof}

As a corollary to the proof, we have:

\begin{corollary}\label{cor:bad_spheres_stay}
 With the same setup as the (ascending version of the) Morse Lemma, if additionally for all vertices $v$ with $s\le h(v)<t$ we have $\widetilde{H}_{m+1}(\alk v)=0$, then the inclusion $Y_{h\ge t} \to Y_{h\ge s}$ induces an injection in $\widetilde{H}_{m+1}$.
 
 With the setup of the descending version, a similar result holds with all the signs reversed and ``ascending'' changed to ``descending''.
\end{corollary}

\begin{proof}
 In the proof of the Morse Lemma, we saw that $Y_{h\ge s}$ is obtained from $Y_{h\ge t}$ by coning off the ascending links of vertices in $Y_{s\le h<t}$, so this is immediate from the Mayer--Vietoris sequence. (The descending version follows similarly.)
\end{proof}

\medskip

One main use of the Morse Lemma is the case $s=-\infty$:

\begin{corollary}\label{cor:morse}
 Let $Y$ be an $(m-1)$-connected affine cell complex with an ascending-type Morse function $(h,f)$. Suppose there exists $N$ such that, for every vertex $v$ of $Y$ with $h(v)<N$, $\alk v$ is $(m-1)$-connected. Then the filtration $(Y_{h\ge t})_{t\in\R}$ is essentially $(m-1)$-connected.
 
 The descending version holds as well, with all the signs reversed.
\end{corollary}

\begin{proof}
 By the Morse Lemma, for any $r\le N$ the inclusion $Y_{h\ge r} \to Y=Y_{h\ge-\infty}$ induces an isomorphism in $\pi_k$ for $k\le m-1$. Since $Y$ is $(m-1)$-connected, so is $Y_{h\ge r}$. Now for any $t$ we just need to choose $s=\min\{N,t\}$ and we get that the inclusion $Y_{h\ge t}\to Y_{h\ge s}$ induces the trivial map in $\pi_k$ for $k\le m-1$, simply because $Y_{h\ge s}$ is $(m-1)$-connected. (The descending version follows similarly.)
\end{proof}

It is less straightforward to use Morse theory to prove that a filtration is \emph{not} essentially $m$-connected. However, given some additional assumptions we can say something.

\begin{proposition}\label{prop:neg_conn}
 Let $Y$ be an affine cell complex and let $(h,f)\colon Y\to \R\times\R$ be an ascending-type Morse function. Suppose there exists $N\in\R$ such that for every vertex $v$ with $h(v)< N$ the ascending link $\alk v$ is $(m-1)$-connected and satisfies $\widetilde{H}_{m+1}(\alk v)=0$. Assume moreover that for all $M\in\R$ there exists a vertex $v$ with $h(v)< M$ such that $\widetilde{H}_m(\alk v)\ne 0$. If $\widetilde{H}_{m+1}(Y)=0$, then the filtration $(Y_{h\ge t})_{t\in\R}$ is not essentially $m$-connected.
 
 If $(h,f)$ is instead descending-type, then a similar result holds with all the signs reversed and ``ascending'' changed to ``descending''.
\end{proposition}

\begin{proof}
 Suppose to the contrary that $(Y_{h\ge t})_{t\in\R}$ is essentially $m$-connected. Say $t< N$, and choose $-\infty<s\le t$ such that the inclusion $Y_{h\ge t} \to Y_{h\ge s}$ induces the trivial map in $\pi_k$ for $k\le m$. Also, since $t< N$, this inclusion induces a surjection in these $\pi_k$ by the Morse Lemma, so in fact $Y_{h\ge s}$ itself is $m$-connected, as are all $Y_{h\ge r}$ for $r\le s$ (for the same reason). Now choose $v$ such that $h(v)< s$ and $\widetilde{H}_m(\alk v)\ne 0$. Since $\widetilde{H}_m(Y_{h\ge r})=0$ for all $r\le s$, Mayer--Vietoris and Corollary~\ref{cor:bad_spheres_stay} say that $\widetilde{H}_{m+1}(Y_{h\ge q})\ne 0$ for any $q\le h(v)$. But this includes $q=-\infty$, and $\widetilde{H}_{m+1}(Y)\ne 0$ contradicts our assumptions. (As usual, the descending version follows similarly.)
\end{proof}

For example, if all the $\alk v$ are homotopy equivalent to (possibly trivial) wedges of $m$-spheres, and for $h(v)$ arbitrarily small are non-trivial wedges of $m$-spheres, then these hypotheses are satisfied.


\section{(Pure) braid groups}\label{sec:gps}

Let $B_n$ denote the $n$-strand braid group. The \emph{standard presentation} for $B_n$ is
$$B_n=\langle s_1,\dots,s_{n-1} \mid s_i s_{i+1}s_i = s_{i+1}s_i s_{i+1} \text{ for } 1\le i\le n-2 \text{ and } s_i s_j=s_j s_i \text{ for } |i-j|>1 \rangle\text{.}$$
Pictorially, $s_i$ is the braid in which the $i$th strand crosses in front of the $(i+1)$st strand. This is called a \emph{positive} crossing, and the reverse is a \emph{negative} crossing. By imposing the additional relations $s_i^2=1$ for all $i$, one obtains the \emph{standard presentation} for the symmetric group $S_n$. Hence there is an epimorphism $\pi \colon B_n\to S_n$ and its kernel is the $n$-strand \emph{pure braid group} $P_n$. We take the action of $S_n$ on $\{1,\dots,n\}$ to be a right action, so the notation $(i)\sigma$ will be common. We will always label the strands of a braid at their tops, by the numbers $1$ through $n$, from left to right. If we then labeled the strands at the bottom, $1$ through $n$ left to right, then each strand would be labeled $i$ at the top and $(i)\pi(x)$ at the bottom for some $i$.

An element of $B_n$ that will be important in all that follows is
$$\Delta\defeq s_1 \cdots s_{n-1} s_1\cdots s_{n-2} \cdots s_1 s_2 s_1 \text{.}$$
This is the element in which each strand crosses in front of every strand to its right, exactly once. See Figure~\ref{fig:delta}. Note that every crossing in $\Delta$ is positive.

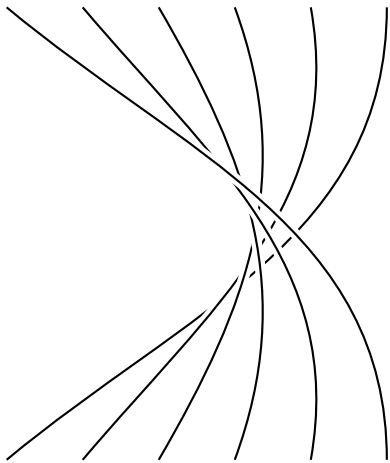
\begin{figure}[htb]
 \begin{tikzpicture}[line width=0.8pt]\centering
  \draw (5,0) to[out=-90, in=40] (0,-6);
  \draw[white, line width=4pt] (4,0) to[out=-80, in=50] (1,-6);
  \draw (4,0) to[out=-80, in=50] (1,-6);
  \draw[white, line width=4pt] (3,0) to[out=-70, in=60] (2,-6);
  \draw (3,0) to[out=-70, in=60] (2,-6);
  \draw[white, line width=4pt] (2,0) to[out=-60, in=70] (3,-6);
  \draw (2,0) to[out=-60, in=70] (3,-6);
  \draw[white, line width=4pt] (1,0) to[out=-50, in=80] (4,-6);
  \draw (1,0) to[out=-50, in=80] (4,-6);
  \draw[white, line width=4pt] (0,0) to[out=-40, in=90] (5,-6);
  \draw (0,0) to[out=-40, in=90] (5,-6);
 \end{tikzpicture}
 \caption{The element $\Delta$ in $P_6$.}\label{fig:delta}
\end{figure}

We collect here some well known facts about $\Delta$. These can be found for example in \cite[Section~1.3.3 and Chapter~6]{kassel08}.

\begin{cit}[$\Delta$ facts]\label{cit:delta_facts}
 The permutation $\pi(\Delta) \in S_n$ is the so called longest word $w_0$, which interchanges $i$ and $n-i+1$, for all $1\le i\le \lfloor n/2\rfloor$. Using the defining relations, one can see that there exists, for any $1\le i\le n-1$, a minimal word representing $\Delta$ and beginning with $s_i$. Also, it is clear that $s_i \Delta=\Delta s_{n-i}$ for all $1\le i\le n-1$. For $n\ge 3$, the center of $B_n$ is infinite cyclic, generated by $\Delta^2$. Since $w_0^2=\id$, the center lies in $P_n$, and in fact the center of $P_n$ also equals $\langle\Delta^2\rangle$ (for all $n$).
\end{cit}

\begin{definition}[Automorphisms]\label{def:auts}
 There are some useful automorphisms of $P_n$ that we discuss now. First, we have the conjugation action of $B_n$ on $P_n$. Later we will consider characters $\omega_{i,j}$ of $P_n$, indexed by the labeling of the strands, and the conjugation action of $B_n$ will permute the indices $i,j$. The other important automorphism is ``look in a mirror''. This takes a braid diagram and sends it to its mirror image. In particular if the automorphism is denoted $\mu$, we have $\mu(s_i)=s_i^{-1}$ for all $i$, and $\mu(\Delta)=\Delta^{-1}$. Note that $\mu$ does not change the labeling of the strands.
\end{definition}

\begin{definition}[Natural projections]\label{def:nat_proj}
 Following \cite{koban15}, we define for any $A\subseteq\{1,\dots,n\}$ an epimorphism $\phi_A \colon P_n \to P_{|A|}$ given by erasing all the strands labeled by numbers not in $A$. Here as before we always label strands $1$ through $n$ from left to right. We call the $\phi_A$ \emph{natural projections}. They are in fact split epimorphisms, with the splitting given by embedding $P_{|A|}$ into $P_n$ by adding the missing strands in the back. See also \cite[Section~1.3.2]{kassel08} for more details.
\end{definition}

\subsection{Characters}\label{sec:chars}

It is easy to abelianize $B_n$ using the standard presentation, and find that $B_n^{ab}\cong \Z$. Hence we have $\Hom(B_n,\R)\cong\R$, and it is noteworthy that this is independent of $n$. In contrast, one can find presentations for $P_n$ (see for example \cite{margalit09}) that reveal that $P_n^{ab}\cong \Z^{\binom{n}{2}}$. Hence $\Hom(P_n,\R)\cong \R^{\binom{n}{2}}$, the dimension of which grows quickly with $n$.

There are bases of these vector spaces that correspond to natural measurements on braids. For $B_n$, a nice basis character, which we will call $\cross$, is the one that reads off the total number of positive crossings of the strands, minus the total number of negative crossings. The character $\cross$ is induced by sending every standard generator of $B_n$ to $1$.

\begin{corollary}\label{cor:braid_sig}
 $\Sigma^\infty(B_n)=S(B_n)$.
\end{corollary}

\begin{proof}
 The points of $S(B_n)$ are $[\cross]$ and $[-\cross]$. Since $\cross(\Delta^2)\ne0$, the result follows from Corollary~\ref{cor:center_survives}.
\end{proof}

In particular Citation~\ref{cit:sig_fin} implies that every coabelian subgroup of $B_n$ (for example $[B_n,B_n]$) is of type $\F_\infty$.

What we actually care about is the pure braid groups; the character sphere $S(P_n)$ has dimension $\binom{n}{2}-1$, so here the situation is much more complicated. A standard basis for the vector space $\Hom(P_n,\R)$ is given by, for each $1\le i<j\le n$, the \emph{winding number characters} $\omega_{i,j}$. The character $\omega_{i,j} \colon P_n \to \Z$ is defined to be
$$\omega_{i,j}\defeq \theta \circ \phi_{\{i,j\}} \text{,}$$
where $\phi_{\{i,j\}}$ is the natural projection $P_n\to P_2$ and $\theta \colon P_2 \to \Z$ is the isomorphism sending $\Delta^2$ to $1$. Pictorially, $\omega_{i,j}\colon P_n\to \Z$ reads off the number of times the $i$th and $j$th strands wind completely around each other. The character $\omega_{i,j}$ can also be denoted $\omega_{j,i}$.

The function $\omega_{i,j}$ makes sense on $B_n$ too, it just is not a homomorphism. Let $\zeta \colon B_2 \to \frac{1}{2}\Z$ be the isomorphism sending $\Delta$ to $\frac{1}{2}$, so $\zeta$ restricted to $P_2$ is $\theta$. For $A\subseteq\{1,\dots,n\}$ let $\psi_A \colon B_n \to B_{|A|}$ be the function (\emph{not} homomorphism) that takes a braid in $B_n$, with strands labeled $1$ through $n$ at the top, and erases all the strands except those labeled by elements of $A$. Now define
$$\omega_{i,j}\defeq \zeta \circ \psi_{\{i,j\}} \text{.}$$
We will use the notation $\omega_{i,j}$ for both the function $B_n\to\frac{1}{2}\Z$ and the homomorphism $P_n\to\Z$ since they coincide on $P_n$. We emphasize that when computing $\omega_{i,j}$ of a braid in $B_n$, we label the strands at the top, not the bottom. Note that for any $x\in B_n$ we have
$$\sum\limits_{1\le i<j\le n}2\omega_{i,j}(x) = \cross(x)\text{.}$$
See Figure~\ref{fig:wind} for some examples of $\omega_{i,j}$-values on a non-pure braid.

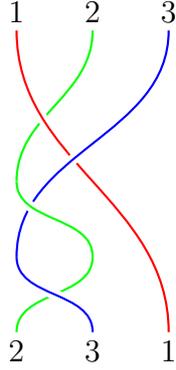
\begin{figure}[htb]
 \begin{tikzpicture}[line width=0.8pt]\centering
  \draw[green] (1,0) to[out=-90, in=90] (0,-2);
  \draw[white, line width=4pt] (0,0) to[out=-90, in=90] (2,-4);
  \draw[red] (0,0) to[out=-90, in=90] (2,-4);
  \draw[white, line width=4pt] (2,0) to[out=-90, in=90] (0,-3);
  \draw[blue] (2,0) to[out=-90, in=90] (0,-3);
  \draw[white, line width=4pt] (0,-2) to[out=-90, in=90] (1,-3) to[out=-90, in=90] (0,-4);
  \draw[green] (0,-2) to[out=-90, in=90] (1,-3) to[out=-90, in=90] (0,-4);
  \draw[white, line width=4pt] (0,-3) to[out=-90, in=90] (1,-4);
  \draw[blue] (0,-3) to[out=-90, in=90] (1,-4);
  \node at (0,.25) {$1$}; \node at (1,.25) {$2$}; \node at (2,.25) {$3$};
  \node at (0,-4.25) {$2$}; \node at (1,-4.25) {$3$}; \node at (2,-4.25) {$1$};
 \end{tikzpicture}
 \caption{We label the strands at the top to compute the $\omega_{i,j}$. The different colors for the strands and the labeling at the bottom is just to make it easier to see how the strands interact. We compute that $\omega_{1,2}=\frac{1}{2}$, $\omega_{1,3}=-\frac{1}{2}$ and $\omega_{2,3}=1$.}\label{fig:wind}
\end{figure}

Despite $\omega_{i,j}\colon B_n \to \frac{1}{2}\Z$ not being a homomorphism, we do have the following:

\begin{observation}\label{obs:twisted_hom}
 Let $x,y\in B_n$ and $1\le i<j\le n$. Then
 $$\omega_{i,j}(xy) = \omega_{i,j}(x) + \omega_{(i)\pi(x),(j)\pi(x)}(y)\text{.}$$
\end{observation}

\begin{proof}
 This is just a matter of keeping track of the labeling of the strands. The strand labeled $i$ at the top of $xy$ is labeled $(i)\pi(x)$ at the top of $y$, and similarly for $j$. Hence the total winding number of strands $i$ and $j$ is their winding number through $x$, plus their winding number through $y$, and we get the desired equation. See Figure~\ref{fig:twist} for an example.
\end{proof}

\begin{figure}[htb]
 \begin{tikzpicture}[line width=0.8pt]\centering
  \draw (1,0) to[out=-90, in=90] (0,-2);
  \draw (2,0) to[out=-90, in=90] (1,-2);
  \draw[white, line width=4pt] (0,0) to[out=-90, in=90] (2,-2);
  \draw (0,0) to[out=-90, in=90] (2,-2);
  \node at (0,.25) {$1$}; \node at (1,.25) {$2$}; \node at (2,.25) {$3$};
  \node at (-.5,-1) {$x$};
  \node at (3.25,-1) {$\omega_{1,3}=1/2$};
  
  \draw[dashed] (-.5,-2.25) -- (2.5,-2.25);
  
  \draw (0,-2.75) -- (0,-4.75);
  \draw (2,-2.75) to[out=-90, in=90] (1,-4.75);
  \draw[white, line width=4pt] (1,-2.75) to[out=-90, in=90] (2,-4.75);
  \draw (1,-2.75) to[out=-90, in=90] (2,-4.75);
  \node at (0,-2.5) {$1$}; \node at (1,-2.5) {$2$}; \node at (2,-2.5) {$3$};
  \node at (-.5,-3.75) {$y$};
  \node at (3.25,-3.75) {$\omega_{3,2}=1/2$};
 \end{tikzpicture}
 \caption{Here $\omega_{1,3}(xy)=1$, $\omega_{1,3}(x)=1/2$ and $\omega_{(1)\pi(x),(3)\pi(x)}(y)=\omega_{3,2}(y)=1/2$, so the equation in Observation~\ref{obs:twisted_hom} is satisfied.}\label{fig:twist}
\end{figure}
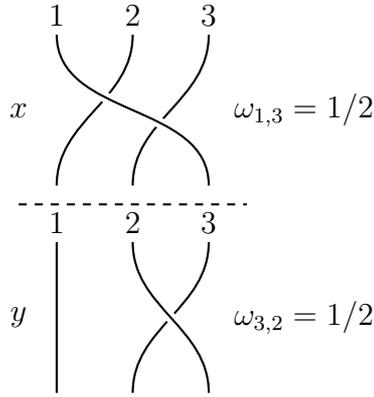

Taking linear combinations of the $\omega_{i,j}$, we can view any character $\chi$ of $P_n$ as a function $B_n \to \R$. If $\chi=\sum a_{i,j}\omega_{i,j}$, denote by $\chi^x$ the character
$$\chi^x \defeq \sum\limits_{1\le i<j\le n} a_{i,j} \omega_{(i)\pi(x),(j)\pi(x)}\text{,}$$
so $\chi(xy)=\chi(x)+\chi^x(y)$ for all $x,y\in B_n$.

Before this subsection we discussed automorphisms of $P_n$. We now know how conjugation by elements of $B_n$ affects characters, namely conjugation by $x$ takes $\chi$ to $\chi^x$ (this is the reason for the notation). The other important automorphism was $\mu$, which is induced by inverting each generator. This takes any character $\chi$ to $-\chi$. In particular since the BNSR-invariants are invariant under automorphisms, we now know that they are closed under taking antipodes:

\begin{observation}[Antipodes]\label{obs:antipodes}
 If $[\chi]\in\Sigma^m(P_n)$ then $[-\chi]\in\Sigma^m(P_n)$. \qed
\end{observation}

We now see more evidence that the element $\Delta$ is important.

\begin{lemma}\label{lem:hom_center}
 Let $\chi \colon P_n \to \R$ be any character, and view $\chi$ as a function from $B_n$ to $\R$ as above. Then for any $x\in B_n$, we have $\chi(x\Delta)=\chi(x)+\chi(\Delta)$.
\end{lemma}

\begin{proof}
 It suffices to check this for $\chi=\omega_{i,j}$. By Observation~\ref{obs:twisted_hom}, we just need to prove that $\omega_{(i)\pi(x),(j)\pi(x)}(\Delta) = \omega_{i,j}(\Delta)$ for all $x\in B_n$ and all $i,j$. Indeed, $\omega_{i,j}(\Delta)=1/2$ for all $i$ and $j$, so we are done.
\end{proof}

\begin{corollary}[$\Delta$ survives]\label{cor:delta_survives}
 Let $\chi \colon P_n \to \R$ be any character, and view $\chi$ as a function from $B_n$ to $\R$ as above. If $\chi(\Delta)\ne 0$ then $[\chi]\in\Sigma^\infty(P_n)$.
\end{corollary}

\begin{proof}
 By Lemma~\ref{lem:hom_center}, $\chi(\Delta^2)=2\chi(\Delta)$, so if this is non-zero then $\chi(Z(P_n))\ne0$, and we are done by Corollary~\ref{cor:center_survives}.
\end{proof}

Since $P_n$ has an $(n-1)$-dimensional compact classifying space, for example the Brady complex \cite{brady01}, Citation~\ref{cit:raise_chi} tells us that $\Sigma^{n-1}(P_n)=\Sigma^\infty(P_n)$. We can do one better than this though:

\begin{theorem}\label{thrm:pb_sigma_ceiling}
 For any $n\in\N$, $\Sigma^{n-2}(P_n)=\Sigma^\infty(P_n)$.
\end{theorem}

\begin{proof}
 First note that the character $\omega_{1,2} \colon P_n \to \Z$ splits, via the map $\Z\to P_n$ sending $1$ to $\Delta^2$, and $\Delta^2$ is central, so $P_n\cong (P_n/Z(P_n)) \times \Z$. (Of course we could have used any $\omega_{i,j}$.) Now, $H\defeq P_n/Z(P_n)$ is isomorphic to the mapping class group of the $(n+1)$-punctured sphere (see for example \cite[Chapter~9]{farb12}), which has cohomological dimension $n-2$ \cite{harer86} and acts freely cocompactly on a contractible $(n-2)$-dimensional complex (see for example \cite[Corollary~1.3]{aramayona14} for an even stronger result). By Citation~\ref{cit:raise_chi} then, $\Sigma^{n-2}(H)=\Sigma^\infty(H)$.

 We now return to $P_n\cong H\times\Z$ itself. Let $\chi\colon P_n\to\R$ be a non-trivial character, and suppose $[\chi]\in\Sigma^{n-2}(P_n)$. If $\chi(Z(P_n))\ne 0$ then $[\chi]\in\Sigma^\infty(P_n)$ by Corollary~\ref{cor:center_survives}, so assume $\chi(Z(P_n))=0$. Then $\chi$ is induced via $P_n\to H$ by a character $\chi'$ of $H$. By Citation~\ref{cit:push_neg}, $[\chi']\in\Sigma^{n-2}(H)$. As we have just seen, this means $[\chi']\in\Sigma^\infty(H)$. Now since $H$ is of type $\F_\infty$ and $P_n$ is a trivial HNN-extension of $H$, we conclude, for example from \cite[Theorem~2.3]{bieri10}, that $[\chi]\in\Sigma^\infty(P_n)$.
\end{proof}


\section{The complex}\label{sec:cpx}

In this section we define a locally compact simplicial complex $X$ on which the braid group $B_n$ acts freely and vertex transitively. Since $P_n$ has finite index in $B_n$, its action on $X$ is cocompact, and hence can be used to ``reveal'' the BNSR-invariants of $P_n$. This complex arises from a \emph{Garside structure} on the symmetric group $S_n$, and has appeared before, see for example \cite{charney04}. For our purposes it will be useful to view $X$ in a particular, slightly unconventional way, so we will start at the beginning and stay mostly self-contained.

Recall the standard group presentation for the braid group $B_n$:
$$\langle s_1,\dots,s_{n-1} \mid s_i s_{i+1}s_i = s_{i+1}s_i s_{i+1} \text{ for } 1\le i\le n-2 \text{ and } s_i s_j=s_j s_i \text{ for } |i-j|>1 \rangle \text{.}$$
Viewing this as a monoid presentation, we get the monoid $B^+_n$ of positive braids, that is, the subset of $B_n$ consisting of elements representable by words in the $s_i$ (and not the $s_i^{-1}$). Pictorially, a \emph{positive braid} is one representable in such a way that every crossing is positive. Note that $\cross \colon B^+_n \to \N_0$ is a monoid homomorphism, and $\cross(p)=0$ for $p\in B^+_n$ if and only if $p=1$, so $B^+_n$ has no non-trivial units.

Define a relation $\le$ on $B_n$ by saying $x\le y$ whenever $xp=y$ for $p\in B^+_n$. Note that left multiplication by $B_n$ preserves this relation. Since $B^+_n$ is a monoid with no non-trivial units, $\le$ is a partial order. Since $s_i\le\Delta$ for all $i$ and $\Delta^2$ is central (Citation~\ref{cit:delta_facts}), every $x\in B_n$ satisfies $x\le \Delta^k$ for some $k>0$. Thus, any two elements of $B_n$ have a common upper bound, i.e., the poset is directed. In particular, the geometric realization $|B_n|$ is a contractible space on which $B_n$ acts freely (the action is free on vertices by construction, and it is easy to see that the stabilizer of a vertex equals the intersection of its vertex stabilizers). It is not, however, cocompact; indeed it is not even finite dimensional. Our next goal is to retract $|B_n|$ down to a $B_n$-invariant subcomplex that is cocompact.

\begin{definition}
 For $x\le y$, write $x\preceq y$ if moreover $y\le x\Delta$.
\end{definition}

These next two results follow quickly from the definition.

\begin{lemma}\label{lem:shift}
 If $x\preceq y$ then $y\preceq x\Delta$ and $y\Delta^{-1} \preceq x$.
\end{lemma}

\begin{proof}
 We are told that $x\le y\le x\Delta$. Say $y=xp$ for $p\in B^+_n$. Then since $s_i \Delta=\Delta s_{n-i}$ for all $1\le i\le n-1$ (see Citation~\ref{cit:delta_facts}), we see that $\Delta^{-1} p \Delta \in B^+_n$. In particular $y\Delta=x\Delta(\Delta^{-1}p\Delta)$ tells us that $x\Delta\le y\Delta$, and so indeed $y\preceq x\Delta$. Now say $x\Delta=yq$ for $q\in B^+_n$. Then since $\Delta q \Delta^{-1} \in B^+_n$, similar to the previous case, the equation $y\Delta^{-1}(\Delta q\Delta^{-1})=x$ tells us that $y\Delta^{-1} \le x$, and so $y\Delta^{-1} \preceq x$.
\end{proof}

\begin{observation}\label{obs:cotrans}
 If $x\le y\le z$ and $x\preceq z$ then $x\preceq y\preceq z$.
\end{observation}

\begin{proof}
 We are told that $z\le x\Delta$. Since $y\le z$ we immediately see that $y\le x\Delta$, so $x\preceq y$. Now we claim that $x\Delta\le y\Delta$, after which we will get $z\le y\Delta$ and hence $y\preceq z$. Choose $p\in B^+_n$ such that $xp=y$. As in the previous proof, $\Delta^{-1} p \Delta \in B^+_n$. Hence $x\Delta(\Delta^{-1} p \Delta)=y\Delta$ tells us that indeed $x\Delta\le y\Delta$.
\end{proof}

We now let $X$ be the subcomplex of $|B_n|$ consisting of those simplices, i.e., chains $x_0<\cdots<x_k$, such that $x_0\preceq x_k$. By Observation~\ref{obs:cotrans}, for such a simplex we have $x_i\preceq x_j$ for all $i$ and $j$, so this property is closed under passing to faces and $X$ is indeed a subcomplex. Using interval notation in the poset $(B_n,\le)$, every simplex of $X$ lies in the realization of an interval of the form $[x,x\Delta]$, and any such interval is finite, so $X$ is locally compact. Also note that $X^{(0)}=|B_n|^{(0)}$, so $B_n$ acts transitively on $X^{(0)}$, which means that $B_n\backslash X$ is compact.

\begin{proposition}
 $X$ is contractible.
\end{proposition}

\begin{proof}
 Any simplex of $|B_n|$ lies in the realization of an interval of the form $|[x,xp]|$ for $x\in B_n$ and $p\in B^+_n$. If this realization is not contained in $X$, then Observation~\ref{obs:cotrans} says $x\not\preceq xp$. Hence we can build up from $X$ to $|B_n|$ by gluing in realizations of such intervals along their relative links, and we claim we can do this in such a way that said relative links are always contractible. This will then imply that the homotopy type never changes, and indeed $X\simeq |B_n|$ is contractible.

 We glue in the intervals $|[x,xp]|$ for $x\not\preceq xp$ in order of increasing $\cross(p)$ value. Hence, when we attach $|[x,xp]|$, we do so along a relative link equal to $|[x,xp)| \cup |(x,xp]|$. This is just the suspension of $|(x,xp)|$, so it suffices to prove that for $x\not\preceq xp$, $|(x,xp)|$ is contractible. To clean up notation, without loss of generality $x=1$, so our assumption that $x\not\preceq xp$ becomes $1\not\preceq p$, which just means that $p\not\le\Delta$.

 Any $q\in B^+_n$ has a unique greatest lower bound with $\Delta$, denoted $q\wedge \Delta$ (see for example \cite[Theorems~6.19 and~6.20]{kassel08} ). Let $f\colon (1,p) \to (1,p)$ be the map $q\mapsto q \wedge \Delta$. It is clear that $f(q)\in [1,p)$ for $q\in (1,p)$, and it is less clear but still true that it is in $(1,p)$. This is because $q\in B^+_n\setminus\{1\}$, and any such $q$ satisfies $1<q\wedge \Delta$ since $\Delta$ admits minimal representations beginning with any positive generator (see Citation~\ref{cit:delta_facts}). Since $f(q)\le q$ and any element $f(q)$ in the image of $f$ satisfies $f(q)\le p\wedge\Delta$, if we can show that $p\wedge\Delta\in (1,p)$ then we will have a conical contraction of $(1,p)$ to a point, by \cite[Section~1.5]{quillen78}. Here is where we use the fact that $p\not\le\Delta$; this implies that $p\wedge \Delta \ne p$, so indeed $p\wedge\Delta \in (1,p)$ and we are done.
\end{proof}

\begin{remark}
 There are various other useful models for classifying spaces of braid groups, e.g., the Brady complex \cite{brady01}. The Brady complex is similar to $X$, but with smaller lattices than our $[1,\Delta]$, namely \emph{non-crossing partition lattices}. The Brady complex has many advantages, for example it has the smallest dimension possible, but for our purposes $X$ is the most germane model, since once characters enter the picture, $\Delta$ becomes very useful.
\end{remark}

\subsection{Character height functions}\label{sec:char_ht}

Recall from Subsection~\ref{sec:chars} that we can view characters $\chi$ of $P_n$ as functions on $B_n$, which might just not be homomorphisms. Now, $B_n=X^{(0)}$, so extending $\chi$ affinely to the simplices of $X$, we get a function $h_\chi \colon X\to\R$. To make the notation cleaner, from now on we will just write $\chi$ for both the character on $P_n$ and the function on $X$. Indeed, $h_\chi$ of a vertex was already also called $\chi$, so calling the entire thing $\chi$ is not even really an abuse of notation.

\begin{observation}
 Let $\chi$ be a character of $P_n$. Viewing $\chi$ as a function on $X$, it is a character height function.
\end{observation}

\begin{proof}
 For a vertex $x\in X$ and a group element $g\in P_n$, since $\pi(g)=\id$, Observation~\ref{obs:twisted_hom} says $\chi(gx)=\chi(g)+\chi(x)$.
\end{proof}

More generally, Observation~\ref{obs:twisted_hom} says $\chi(xy)=\chi(x)+\chi^x(y)$ for $x,y\in B_n$ (with notation explained after the proof of Observation~\ref{obs:twisted_hom}). In particular, if $x$ is a vertex of $X$ (so really an element of $B_n$) and $p\in [\Delta^{-1},1)\cup(1,\Delta]$, so $xp$ is an adjacent vertex to $x$, then $\chi(xp)=\chi(x)+\chi^x(p)$. We see that $\chi(xp)>\chi(x)$ if and only if $\chi^x(p)>0$. Since $\chi^x(1)=0$, this can be phrased as saying that $xp$ is $\chi$-ascending relative $x$ if and only if $p$ is $\chi^x$-ascending relative $1$.

\subsection{The weak Bruhat lattice}\label{sec:bruhat}

The vertex links in $X$ are important to understand, and they are related to the \emph{weak Bruhat lattice} $\W_n$ of the symmetric group $S_n$. This section is devoted to $\W_n$.

We will write elements $\sigma$ of $S_n$ as brackets $\sigma=[k_1,\dots,k_n]$, where this denotes that, for each $1\le i\le n$, $\sigma$ sends $i$ to $j$ such that $k_j=i$. In a braid $x$, where we label the strands from $1$ to $n$ at the top, if we propagate each label to the bottom of its strand, and call the new order $k_1,\dots,k_n$ from left to right, then we get $\pi(x)=[k_1,\dots,k_n]$. For example, going back to the braid $x$ from Figure~\ref{fig:wind}, we have $\pi(x)=[2,3,1]$, which is the permutation taking $2$ to $1$, $3$ to $2$, and $1$ to $3$.

\begin{definition}[Weak Bruhat order]\label{def:wbo}
 Let $\sigma,\tau\in S_n$. Define a relation $\le$ by saying that $\sigma\le \tau$ whenever, for all $1\le i<j\le n$, if $(i)\sigma>(j)\sigma$ then $(i)\tau>(j)\tau$ (recall we use right actions). This is clearly reflexive, antisymmetric and transitive, so $(S_n,\le)$ is a poset. We call $\le$ the \emph{weak Bruhat order}. The poset $(S_n,\le)$ is in fact a lattice \cite{bjoerner90} with minimum $[1,\dots,n]$ and maximum $[n,\dots,1]$.
\end{definition}

We denote the geometric realization $|(S_n,\le)|$ by $\W_n$, and call it the \emph{weak Bruhat lattice}. The \emph{proper part} of $\W_n$, denoted $\PW_n$, is the subcomplex supported on those vertices other than $[1,\dots,n]$ and $[n,\dots,1]$. The reason we are so interested in the weak Bruhat order is the following:

\begin{lemma}
 The interval $[1,\Delta]$ in the poset $(B_n,\le)$ is isomorphic as a poset to $(S_n,\le)$, the symmetric group with the weak Bruhat order.
\end{lemma}

\begin{proof}
 Let $\pi \colon [1,\Delta] \to S_n$ be the restriction of the usual projection $B_n \to S_n$. This restriction is bijective, by \cite[Lemma~6.24]{kassel08}, so we just need to show it is a poset map. Suppose $x\le y$ in $B_n$. We need to show that $\pi(x)\le \pi(y)$. Say $xp=y$, so for any $1\le i<j\le n$ we have $\omega_{i,j}(y)=\omega_{i,j}(x) + \omega_{(i)\pi(x),(j)\pi(y)}(p)$ by Observation~\ref{obs:twisted_hom}. If $(i)\pi(x)>(j)\pi(x)$ then since $1\le x\le \Delta$, we know $\omega_{i,j}(x)=1/2$. Also, since $1\le y\le\Delta$ we know $\omega_{i,j}(y)$ is either $0$ or $1/2$, so since $\omega_{(i)\pi(x),(j)\pi(y)}(p)\ge0$ this implies $\omega_{i,j}(y)=1/2$. Hence $(i)\pi(y)>(j)\pi(y)$ and indeed $\pi(x)\le \pi(y)$.
\end{proof}

In the future we may suppress this isomorphism, and use the language of the weak Bruhat order when talking about the interval $[1,\Delta]$. Since we are using $\le$ to denote the order in both places anyway, this may even pass unnoticed.

There is a family of subcomplexes of $\W_n$ that will be important later. For $1\le i<j\le n$, the \emph{proper $(i,j)$-reversing subcomplex} $\Rev_n(i,j)$ is the subcomplex of $\PW_n$ supported on those vertices $\sigma$ with $(i)\sigma>(j)\sigma$. In bracket notation, if $\sigma=[k_1,\dots,k_n]$ then $\sigma\in \Rev_n(i,j)$ if and only if $j$ comes before $i$ in the list $k_1,\dots,k_n$. If we allowed $[n,\dots,1]$, this would be a cone point for $\Rev_n(i,j)$, but since $\Rev_n(i,j)$ lies in $\PW_n$, computing its homotopy type requires some work.

\begin{lemma}\label{lem:close_reverse}
 For $1\le i<j\le n$ with $1<i$ or $j<n$, $\Rev_n(i,j)$ is contractible.
\end{lemma}

\begin{proof}
 We will cover $\Rev_n(i,j)$ by contractible subcomplexes, namely stars of certain vertices, and then use the Nerve Lemma \cite[Lemma~1.2]{bjoerner94}. Call a vertex of $\Rev_n(i,j)$ \emph{minimal} if it is minimal under the ordering of $\PW_n$ restricted to $\Rev_n(i,j)$. Clearly any simplex of $\Rev_n(i,j)$ lies in the star of some minimal vertex. Note that if $x\in \Rev_n(i,j)$ and $x\le y$ for $y\in\PW_n$, then also $y\in\Rev_n(i,j)$. In particular, $\Rev_n(i,j)$ is closed under taking upper bounds not equal to $[n,\dots,1]$. Thanks to this, if a collection of stars of minimal vertices has a non-empty intersection, then their intersection is itself a star, namely the star of the join of these vertices. In particular, the Nerve Lemma says that $\Rev_n(i,j)$ is homotopy equivalent to the nerve of the covering by stars of minimal vertices.
 
 Note that $\sigma$ is minimal if and only if $\sigma=[k_1,\dots,k_r,j,i,k_{r+3},\dots,k_n]$ for some $k_1<\cdots<k_r<j$ and $i<k_{r+3}<\cdots<k_n$. Hence, the join in $\W_n$ of all the minimal vertices of $\Rev_n(i,j)$ is $[1,\dots,i-1,j,j-1,\dots,i+1,i,j+1,\dots,n]$. Since $1<i$ or $j<n$, this lies in $\PW_n$ and hence in $\Rev_n(i,j)$. In particular, the nerve of the covering is a simplex, implying $\Rev_n(i,j)$ is contractible.
\end{proof}

\begin{lemma}\label{lem:far_reverse}
 For $n\ge3$, $\Rev_n(1,n)$ is homotopy equivalent to an $(n-3)$-sphere.
\end{lemma}

\begin{proof}
 As in the previous proof, we will cover $\Rev_n(1,n)$ by stars of certain vertices and use the Nerve Lemma. This time we want to use \emph{maximal} vertices. The maximal vertices of $\PW_n$ are those of the form $[n,\dots,i+2,i,i+1,i-1,\dots,1]$. That is, for each $1\le i< n$ we have a maximal vertex $v_i$ in which every pair of entries is reversed except for $(i,i+1)$. Since $n\ge3$, the pair $(1,n)$ is not of this form, so all these maximal vertices of $\PW_n$ lie in $\Rev_n(1,n)$. Clearly any simplex of $\Rev_n(1,n)$ lies in the star of some maximal vertex, so $\Rev_n(1,n)$ is covered by the stars of its maximal vertices, of which there are $n-1$. If a collection of such stars has non-empty intersection in $\Rev_n(1,n)$, then it equals the star of the meet of the relevant maximal vertices. Hence by the Nerve Lemma, $\Rev_n(1,n)$ is homotopy equivalent to the nerve of this covering.
 
 Now, given any $n-2$ maximal vertices, say all of them except $v_i$, the meet of these is $[i+1,\dots,n,1,\dots,i]$. Since this lies in $\Rev_n(1,n)$, every proper subset of vertices of the nerve spans a simplex. However, the collection of all the vertices does not span a simplex, since the meet in $\W_n$ of all the maximal vertices is $[1,\dots,n]$, and this does not lie in $\Rev_n(1,n)$, in fact it does not even lie in $\PW_n$. Hence the nerve is the boundary of an $(n-2)$-simplex, which is homotopy equivalent to an $(n-3)$-sphere.
\end{proof}


\section{Main results}\label{sec:results}

Since the action of $P_n$ on $X$ is free and cocompact, and $X$ is contractible, for any $[\chi]\in S(P_n)$ we know that $[\chi]\in\Sigma^m(P_n)$ if and only if the filtration $(X_{\chi\ge t})_{t\in\R}$ is essentially $(m-1)$-connected. We now encode $\chi$ into a Morse function as follows. Consider the following total ordering $\lessdot$ on the integers:
$$\cdots\lessdot 3 \lessdot 2 \lessdot 1 \lessdot \cdots \lessdot -3 \lessdot -2 \lessdot -1 \lessdot 0 \text{.}$$
In words, the non-positive numbers are ordered like normal, but the positive numbers are in the reverse order, and are all less than every non-positive number. Recall the ``number of crossings'' homomorphism $\cross \colon B_n \to \Z$. We will write $\dot{\cross}$ for the function $\cross$ with its outputs ordered according to $\lessdot$. As functions of sets, $\cross$ and $\dot{\cross}$ are identical, the dot is just to emphasize the unusual ordering of the outputs. We will also write $\dot{\Z}$ for the set $\Z$ with the ordering $\lessdot$.

Extend $\dot{\cross}$ to a map $X\to\R$ by affinely extending it to the simplices. Here we have fixed some order-preserving embedding of $\dot{\Z}$ into $\R$, but it will not matter what it is precisely. Consider the lexicographically ordered function
$$(\chi,\dot{\cross}) \colon X \to \R \times \R$$
on $X$.

\begin{lemma}
 $(\chi,\dot{\cross})$ is an ascending-type Morse function.
\end{lemma}

\begin{proof}
 By construction, $\chi$ and $\dot{\cross}$ are affine on cells. Since $P_n$ acts cocompactly, $\chi(v)-\chi(w)$ takes only finitely many values on pairs of adjacent vertices $(v,w)$, so we can choose $\varepsilon$ to be less than all of the non-zero such values. Then since $\dot{\cross}$ always takes different values on adjacent vertices, and since the outputs of $\dot{\cross}$ are inversely well ordered by construction, indeed $(\chi,\dot{\cross})$ is an ascending-type Morse function.
\end{proof}

For a vertex $x$ and a neighboring vertex $y$, $y$ is in the ascending link of $x$ if and only if either $\chi(y)>\chi(x)$, or else $\chi(y)=\chi(x)$ and $\dot{\cross}(y)\gtrdot\dot{\cross}(x)$. If $\cross(x)$ is positive (that is, $\cross(x)>0$ in the usual ordering), then $\dot{\cross}(y)\gtrdot\dot{\cross}(x)$ is equivalent to $\cross(y)<\cross(x)$. If $\cross(x)$ is non-positive, then $\dot{\cross}(y)\gtrdot\dot{\cross}(x)$ is equivalent to $0\ge \cross(y)>\cross(x)$.

Recall from Corollary~\ref{cor:delta_survives} that if $\chi(\Delta)\ne0$ then $[\chi]\in\Sigma^\infty(P_n)$, so this case is finished. For the rest of the section, we will assume that $\chi(\Delta)=0$. In particular, in this case Lemma~\ref{lem:hom_center} tells us that
\begin{eqnarray}\label{eq:flat_shift}
 \chi(x\Delta)=\chi(x)
\end{eqnarray}
for all $x\in B_n$.

\begin{lemma}\label{lem:peeling_down}
 If $\cross(x)$ is positive then $\alk x$ is homotopy equivalent to $\alk x \cap [x\Delta^{-1},x)$, which is contractible.
\end{lemma}

\begin{proof}
 Since $\cross(x)$ is positive and $\cross(x\Delta^{-1})<\cross(x)$, we get that $\dot{\cross}(x\Delta^{-1})\gtrdot\dot{\cross}(x)$. Also, $\chi(x\Delta^{-1})=\chi(x)$ by Equation~\eqref{eq:flat_shift}, so $x\Delta^{-1}\in \alk x$. Hence the intersection $\alk x \cap [x\Delta^{-1},x)$ is a contractible cone on $x\Delta^{-1}$.
 
 Now we need to show the homotopy equivalence. Let $L\defeq \alk x$ and $N\defeq \alk x \cap [x\Delta^{-1},x)$. Since $L$ is finite and $\cross$ takes distinct values on adjacent vertices, $\cross$ is a (classical) Morse function on $L$ (here we use the usual ordering, i.e., we are using $\cross$ and not $\dot{\cross}$). By definition, $N$ is the sublevel set $L_{\cross<\cross(x)}$. Hence by Corollary~\ref{cor:morse}, it suffices to show that for any vertex $y$ of $L$ with $\cross(y)>\cross(x)$, the $\cross$-descending link $\lk_L^{\cross\!\downarrow} y$ of $y$ in $L$ is contractible. For such a $y$, we claim that $\lk_L^{\cross\!\downarrow} y$ is a contractible cone on $y\Delta^{-1}$. First we need to check that $y\Delta^{-1}$ actually lies in $\lk_L^{\cross\!\downarrow} y$. Since $y\in L$ and $\cross(y)>\cross(x)$, we have $x\prec y$, so $y\Delta^{-1} \prec x$ by Lemma~\ref{lem:shift}. This shows $y\Delta^{-1}\in \lk_X x$, but we also need it to be $(\chi,\dot{\cross})$-ascending. For this note that, since $\cross(y)>\cross(x)>0$, we know $\dot{\cross}(x)\gtrdot\dot{\cross}(y)$, and so for $y$ to be in $L$ we must have $\chi(y)>\chi(x)$, and hence $\chi(y\Delta^{-1})>\chi(x)$. We conclude that $y\Delta^{-1}\in L$. Since $\cross(y\Delta^{-1})<\cross(y)$, we have $y\Delta^{-1}\in\lk_L^{\cross\!\downarrow} y$. Now we claim any $z\in \lk_L^{\cross\!\downarrow} y$ satisfies $y\Delta^{-1}\preceq z$. Indeed, such a $z$ satisfies $z\prec y$, so this follows from Lemma~\ref{lem:shift}. We conclude that $\lk_L^{\cross\!\downarrow} y$ is a contractible cone on $y\Delta^{-1}$.
\end{proof}

\begin{lemma}\label{lem:peeling_up}
 If $\cross(x)$ is non-positive then $\alk x$ is homotopy equivalent to $\alk x \cap (x,x\Delta]$.
\end{lemma}

\begin{proof}
 This follows by a somewhat parallel argument to the previous proof. Let $L\defeq \alk x$ and let $P\defeq \alk x \cap (x,x\Delta]$. Consider $\cross$ as a classical Morse function on $L$, so $P$ is the superlevel set $L_{\cross>\cross(x)}$. We need to show that for any vertex $y$ of $L$ with $\cross(y)<\cross(x)$, the $\cross$-ascending link $\lk_L^{\cross\!\uparrow} y$ of $y$ in $L$ is contractible. We claim it is a cone on $y\Delta$. First, $y\prec x$ so $x\preceq y\Delta$ and $y\Delta \in \lk_X x$. Moreover, since $\cross(y)<\cross(x)\le 0$, we have $\dot{\cross}(y)\lessdot\dot{\cross}(x)$, so for $y$ to lie in $L$ we must have $\chi(y)>\chi(x)$. Hence also $\chi(y\Delta)>\chi(x)$ and so $y\Delta \in \lk_L^{\cross\!\uparrow} y$. Lastly, for any $z\in \lk_L^{\cross\!\uparrow} y$ we have $y\prec z$, so $z\preceq y\Delta$, and indeed $\lk_L^{\cross\!\uparrow} y$ is a cone on $y\Delta$, hence contractible.
\end{proof}

There is a conspicuous difference between these two lemmas, which is worth pointing out explicitly, namely, when $\cross(x)$ is non-positive, $\alk x \cap (x,x\Delta]$ could fail to be contractible. The asymmetry comes from the conditions to be ascending: when $\cross(x)$ is positive, a vertex $y\in \lk_X x$ is ascending if either $\chi(y)>\chi(x)$, or $\cross(y)<\cross(x)$, whereas when $\cross(x)$ is non-positive, $y$ is ascending if either $\chi(y)>\chi(x)$, or $0\ge\cross(y)>\cross(x)$. Hence, where $x\Delta^{-1}$ served as a cone point for $\alk x \cap [x\Delta^{-1},x)$, now in $\alk x \cap (x,x\Delta]$ we might not even have $x\Delta$ to use as a cone point, e.g., if $\cross(x\Delta)>0$.

Let us give the complex $P$ from the last proof a more official name.

\begin{definition}[Positive ascending link]
 Define the \emph{positive ascending link} $\palk x$ of $x$ to be $\palk x \defeq \alk x \cap (x,x\Delta]$.
\end{definition}

\begin{corollary}\label{cor:cible_palk}
 If $\cross(x\Delta)\le0$ then $\alk x$ is contractible.
\end{corollary}

\begin{proof}
 Since $\cross(x)\le \cross(x\Delta)$, we know $\cross(x)\le 0$. In this case, Lemma~\ref{lem:peeling_up} says $\alk x\simeq \palk x$, so our goal is to prove that $\palk x$ is contractible. Since $\cross(x\Delta)\le0$, and since $\chi(x\Delta)=\chi(x)$ by Equation~\eqref{eq:flat_shift}, $\cross(x\Delta)\in \palk x$. But this means $\palk x$ is a cone on $x\Delta$, and so is contractible.
\end{proof}

Thanks to Lemma~\ref{lem:peeling_down} and Corollary~\ref{cor:cible_palk}, the only time $\alk x$ might fail to be contractible is when $\cross(x)\le 0<\cross(x\Delta)$. In this case Lemma~\ref{lem:peeling_up} says that $\alk x\simeq \palk x$, which equals the ``$\chi$-ascending part'' of $(x,x\Delta)$ together with some amount of its ``$\chi$-flat part'' added.

\medskip

Here is our main application of all this setup:

\begin{proposition}\label{prop:one_pos}
 Let $\chi=\sum a_{i,j} \omega_{i,j}$ be a character of $P_n$. Suppose that $\sum a_{i,j}=0$, so $\chi(\Delta)=0$, and that either exactly one $a_{i,j}$ is positive, or exactly one $a_{i,j}$ is negative. Then $[\chi]\not\in\Sigma^{n-2}(P_n)$. If moreover none of the $a_{i,j}$ are zero, then $[\chi]\in\Sigma^{n-3}(P_n)$.
\end{proposition}

\begin{proof}
 Any such $\chi$ is the limit of a sequence of characters of the same form, and such that no $a_{i,j}$ is zero. Since $S(P_n)\setminus\Sigma^{n-2}(P_n)$ is closed, we may assume without loss of generality that no $a_{i,j}$ is zero, and then both statements in the theorem will follow if we show that $[\chi]\in\Sigma^{n-3}(P_n)\setminus\Sigma^{n-2}(P_n)$. By Observation~\ref{obs:antipodes}, we may assume without loss of generality that exactly one $a_{i,j}$ is positive (and the others are negative). Say the positive one is $a_{k,\ell}$.
 
 We inspect ascending links $\alk x$ with respect to $(\chi,\dot{\cross})$. Thanks to our assumptions, no proper non-empty subset of $\{a_{i,j}\mid 1\le i<j\le n\}$ sums to zero, so the only vertices of $\lk_X x$ with the same $\chi$-value as $x$ are $x\Delta$ and $x\Delta^{-1}$. By Lemma~\ref{lem:peeling_down} and Corollary~\ref{cor:cible_palk}, $\alk x$ will be contractible unless $\cross(x)\le 0 < \cross(x\Delta)$, so we can assume these bounds hold. In this case, Lemma~\ref{lem:peeling_up} says $\alk x \simeq \palk x$, but we also have $\dot{\cross}(x\Delta)\lessdot \dot{\cross}(x)$, so $\palk x$ is just the subcomplex of $(x,x\Delta)$ supported on those vertices $y$ with $\chi(y)>\chi(x)$. Call this the \emph{$\chi$-ascending part} of $(x,x\Delta)$. (There is no ``$\chi$-flat part'' thanks to our assumptions.)
 
 For $y\in (x,x\Delta)$, say with $y=xp$ for $p\in B^+_n$, we have that $\chi(y)>\chi(x)$ if and only if $\chi^x(p)>0$, as discussed in Subsection~\ref{sec:char_ht}. Hence the $\chi$-ascending part of $(x,x\Delta)$ is isomorphic to the $\chi^x$-ascending part of $(1,\Delta)$. Since $\chi^x=\sum a_{i,j}\omega_{(i)\pi(x),(j)\pi(x)}$, the lone positive coefficient $a_{k,\ell}$ in $\chi^x$ is on $\omega_{(k)\pi(x),(\ell)\pi(x)}$. Let $p$ and $q$ denote $(k)\pi(x)$ and $(\ell)\pi(x)$, in whichever order gives us $p<q$. Thus a vertex of $(1,\Delta)$ is $\chi^x$-ascending if and only if it is $\omega_{p,q}$-ascending. If we view $(1,\Delta)$ as the proper part $\PW_n$ of the weak Bruhat lattice $\W_n$, then this is precisely the subcomplex $\Rev_n(p,q)$ of $\PW_n$. By Lemma~\ref{lem:close_reverse}, this is contractible unless $p=1$ and $q=n$, and then it is homotopy equivalent to $S^{n-3}$. In particular, $\alk x$ is always $(n-4)$-connected, so Corollary~\ref{cor:morse} says $[\chi]\in\Sigma^{n-3}(P_n)$. 
 
 Now, the case when $\alk x \simeq S^{n-3}$ happens precisely for those $x$ for which $\{(k)\pi(x),(\ell)\pi(x)\}=\{1,n\}$, and it is clear that for any $M\in\R$ there exists $x$ with $\chi(x)<M$ such that $\{(k)\pi(x),(\ell)\pi(x)\}=\{1,n\}$. Also, $\widetilde{H}_{n-2}(\alk x)=0$ for all $x$, so Proposition~\ref{prop:neg_conn} tells us that $[\chi]\not\in\Sigma^{n-2}(P_n)$.
\end{proof}

For $3\le m\le n$, consider the character
$$\chi^m_n\defeq \left(\sum\limits_{\stackrel{1\le i<j\le m}{(i,j)\ne(1,2)}} \omega_{i,j}\right) - \left(\binom{m}{2}-1\right)\omega_{1,2} \text{.}$$
Note that sum of the coefficients in $\chi^m_n$ is zero, exactly one coefficient is negative, and if $m=n$ then none of the coefficients is zero. In particular Proposition~\ref{prop:one_pos} tells us that $[\chi^m_m]\in\Sigma^{m-3}(P_m)\setminus\Sigma^{m-2}(P_m)$ and $[\chi^m_n]\not\in\Sigma^{n-2}(P_n)$. We can do better than this though. To prove our main theorem, that $\Sigma^{m-2}(P_n)\subseteq \Sigma^{m-3}(P_n)$ is a proper inclusion for all $3\le m\le n$, we will prove that $[\chi^m_n]\in \Sigma^{m-3}(P_n)\setminus \Sigma^{m-2}(P_n)$.

\begin{theorem}[Separation]\label{thrm:separation}
 For any $3\le m\le n$, the inclusion $\Sigma^{m-2}(P_n)\subseteq \Sigma^{m-3}(P_n)$ is proper. Explicitly, $[\chi^m_n]\in \Sigma^{m-3}(P_n)\setminus \Sigma^{m-2}(P_n)$.
\end{theorem}

\begin{proof}
 For the negative statement, note that $\chi^m_n = \chi^m_m \circ \phi_{\{1,\dots,m\}}$, so Citation~\ref{cit:push_neg} and Proposition~\ref{prop:one_pos} tell us that $[\chi^m_n]\not\in\Sigma^{m-2}(P_n)$. For the positive statement, it is known that the kernel of $\phi_{\{1,\dots,n-1\}}\colon P_n\to P_{n-1}$ is isomorphic to the free group $F_{n-1}$ \cite[Theorem~1.16]{kassel08}, which is of type $\F_\infty$, so Lemma~\ref{lem:push_pos} and Observation~\ref{obs:antipodes} say that when a discrete character on $P_n$ is induced from one on $P_{n-1}$, it inherits the latter's positive BNSR-invariant properties. By induction, when a discrete character on $P_n$ is induced from one on $P_m$ for any $m<n$, it inherits the latter's positive BNSR-invariant properties. Since $[\chi^m_m] \in \Sigma^{m-3}(P_m)$ by Proposition~\ref{prop:one_pos}, we conclude that $[\chi^m_n] \in \Sigma^{m-3}(P_n)\setminus \Sigma^{m-2}(P_n)$.
\end{proof}

\begin{corollary}\label{cor:separation}
 For any $3\le m\le n$, $\ker(\chi^m_n)$ is of type $\F_{m-3}$ but not $\F_{m-2}$.
\end{corollary}

\begin{proof}
 This is immediate from Citation~\ref{cit:sig_fin}, Observation~\ref{obs:antipodes} and Theorem~\ref{thrm:separation}.
\end{proof}

We have thus found examples of coabelian (even ``cocyclic'') subgroups of $P_n$ with every possible finiteness length, namely $0$ through $n-2$.

\medskip

Here is a nice, easy-to-state result in a related vein, which follows by combining our results with the full computation of $\Sigma^1(P_n)$ in \cite{koban15}.

\begin{corollary}
 For any $n\ge4$, $\ker(\omega_{1,2}-\omega_{3,4})\le P_n$ is finitely generated but not finitely presentable.
\end{corollary}

\begin{proof}
 It suffices by Citation~\ref{cit:sig_fin} and Observation~\ref{obs:antipodes} to prove that $[\omega_{1,2}-\omega_{3,4}]\in\Sigma^1(P_n)\setminus \Sigma^2(P_n)$. That it lies in $\Sigma^1$ follows from the complete computation of $\Sigma^1(P_n)$ done in \cite{koban15}. Now we claim that $[\omega_{1,2}-\omega_{3,4}]\not\in\Sigma^2(P_n)$. Using the natural projection $\phi_{\{1,2,3,4\}} \colon P_n\to P_4$, it suffices by Citation~\ref{cit:push_neg} to prove this in the $n=4$ case, but this is immediate from Proposition~\ref{prop:one_pos}.
\end{proof}

\begin{remark}
 Ideally one would like a complete computation of $\Sigma^m(P_n)$ for all $m$ and $n$. Using our setup, if one can show that, for any $x\in B_n$ and any $0\le k< \cross(\Delta)$, the subcomplex of $\PW_n$ supported on those vertices $p$ with either $\chi^x(p)>0$, or $\chi^x(p)=0$ and $\cross(p)\le k$, is homotopy equivalent to a wedge of $m$-spheres, then we could conclude that $[\chi]\in\Sigma^{m-1}(P_n)\setminus \Sigma^m(P_n)$. However, for characters $\chi$ other than those of the type considered here, it is unclear at present how to analyze the homotopy types of these subcomplexes. Also, have found that this method will not always work, for example the Morse function $(\chi,\dot{\cross})$ cannot fully recover the results of \cite{koban15} on $\Sigma^1(P_n)$. Namely, for $\chi$ in a ``$P_4$-circle'', so $[\chi]\not\in\Sigma^1(P_n)$, we have computed that there exist ascending links homotopy equivalent to $S^0$ and others homotopy equivalent to $S^1$, so Proposition~\ref{prop:neg_conn} does not apply.
\end{remark}

\bibliographystyle{amsalpha}

\begin{thebibliography}{BLV{\v{Z}}94}

\bibitem[AMP14]{aramayona14}
J.~Aramayona and C.~Mart{\'{\i}}nez-P{\'e}rez, \emph{The proper geometric
  dimension of the mapping class group}, Algebr. Geom. Topol. \textbf{14}
  (2014), no.~1, 217--227.

\bibitem[BB97]{bestvina97}
M.~Bestvina and N.~Brady, \emph{Morse theory and finiteness properties of
  groups}, Invent. Math. \textbf{129} (1997), no.~3, 445--470.

\bibitem[BEZ90]{bjoerner90}
A.~Bj{\"o}rner, P.~H. Edelman, and G.~M. Ziegler, \emph{Hyperplane arrangements
  with a lattice of regions}, Discrete Comput. Geom. \textbf{5} (1990), no.~3,
  263--288.

\bibitem[BG99]{bux99}
K.-U. Bux and C.~Gonzalez, \emph{The {B}estvina-{B}rady construction revisited:
  geometric computation of {$\Sigma$}-invariants for right-angled {A}rtin
  groups}, J. London Math. Soc. (2) \textbf{60} (1999), no.~3, 793--801.

\bibitem[BGK10]{bieri10}
R.~Bieri, R.~Geoghegan, and D.~H. Kochloukova, \emph{The {S}igma invariants of
  {T}hompson's group {$F$}}, Groups Geom. Dyn. \textbf{4} (2010), no.~2,
  263--273.

\bibitem[BH99]{bridson99}
M.~R. Bridson and A.~Haefliger, \emph{Metric spaces of non-positive curvature},
  Die Grundlehren der Mathematischen Wissenschaften, vol. 319, Springer, 1999.

\bibitem[BLV{\v{Z}}94]{bjoerner94}
A.~Bj{\"o}rner, L.~Lov{\'a}sz, S.~T. Vre{\'c}ica, and R.~T.
  {\v{Z}}ivaljevi{\'c}, \emph{Chessboard complexes and matching complexes}, J.
  London Math. Soc. (2) \textbf{49} (1994), no.~1, 25--39.

\bibitem[BNS87]{bieri87}
R.~Bieri, W.~D. Neumann, and R.~Strebel, \emph{A geometric invariant of
  discrete groups}, Invent. Math. \textbf{90} (1987), no.~3, 451--477.

\bibitem[BR88]{bieri88}
R.~Bieri and B.~Renz, \emph{Valuations on free resolutions and higher geometric
  invariants of groups}, Comment. Math. Helv. \textbf{63} (1988), no.~3,
  464--497.

\bibitem[Bra01]{brady01}
T.~Brady, \emph{A partial order on the symmetric group and new {$K(\pi,1)$}'s
  for the braid groups}, Adv. Math. \textbf{161} (2001), no.~1, 20--40.

\bibitem[Bux04]{bux04}
K.-U. Bux, \emph{Finiteness properties of soluble arithmetic groups over global
  function fields}, Geom. Topol. \textbf{8} (2004), 611--644 (electronic).

\bibitem[CMW04]{charney04}
R.~Charney, J.~Meier, and K.~Whittlesey, \emph{Bestvina's normal form complex
  and the homology of {G}arside groups}, Geom. Dedicata \textbf{105} (2004),
  171--188.

\bibitem[FM12]{farb12}
B.~Farb and D.~Margalit, \emph{A primer on mapping class groups}, Princeton
  Mathematical Series, vol.~49, Princeton University Press, Princeton, NJ,
  2012.

\bibitem[Har86]{harer86}
J.~L. Harer, \emph{The virtual cohomological dimension of the mapping class
  group of an orientable surface}, Invent. Math. \textbf{84} (1986), no.~1,
  157--176.

\bibitem[KMM15]{koban15}
N.~Koban, J.~McCammond, and J.~Meier, \emph{The {BNS}-invariant for the pure
  braid groups}, Groups Geom. Dyn. (2015), To appear. arXiv:1306.4046.

\bibitem[KT08]{kassel08}
C.~Kassel and V.~Turaev, \emph{Braid groups}, Graduate Texts in Mathematics,
  vol. 247, Springer, New York, 2008, With the graphical assistance of Olivier
  Dodane.

\bibitem[Mei97]{meinert97}
H.~Meinert, \emph{Actions on {$2$}-complexes and the homotopical invariant
  {$\Sigma^2$} of a group}, J. Pure Appl. Algebra \textbf{119} (1997), no.~3,
  297--317.

\bibitem[MM09]{margalit09}
D.~Margalit and J.~McCammond, \emph{Geometric presentations for the pure braid
  group}, J. Knot Theory Ramifications \textbf{18} (2009), no.~1, 1--20.

\bibitem[MMV98]{meier98}
J.~Meier, H.~Meinert, and L.~VanWyk, \emph{Higher generation subgroup sets and
  the {$\Sigma$}-invariants of graph groups}, Comment. Math. Helv. \textbf{73}
  (1998), no.~1, 22--44.

\bibitem[MMV01]{meier01}
\bysame, \emph{On the {$\Sigma$}-invariants of {A}rtin groups}, Topology Appl.
  \textbf{110} (2001), no.~1, 71--81, Geometric topology and geometric group
  theory (Milwaukee, WI, 1997).

\bibitem[Qui78]{quillen78}
D.~Quillen, \emph{Homotopy properties of the poset of nontrivial
  {$p$}-subgroups of a group}, Adv. in Math. \textbf{28} (1978), no.~2,
  101--128.

\bibitem[WZ15]{witzel15}
S.~Witzel and M.~C.~B. Zaremsky, \emph{The {$\Sigma$}-invariants of
  {T}hompson's group {$F$}, via {M}orse theory}, Submitted. arXiv:1501.06682,
  2015.

\bibitem[Zar15]{zaremsky15}
M.~C.~B. Zaremsky, \emph{On the {$\Sigma$}-invariants of generalized {T}hompson
  groups and {H}oughton groups}, Submitted. arXiv:1502.02620, 2015.

\end{thebibliography}
\providecommand{\bysame}{\leavevmode\hbox to3em{\hrulefill}\thinspace}
\providecommand{\MR}{\relax\ifhmode\unskip\space\fi MR }
\providecommand{\MRhref}[2]{%
  \href{http://www.ams.org/mathscinet-getitem?mr=#1}{#2}
}
\providecommand{\href}[2]{#2}

\end{document}